\documentclass{amsart} 
\usepackage{amssymb} 
\usepackage{amsthm}  

\theoremstyle{plain} 
\newtheorem{Theorem}{Theorem}[section] 
\theoremstyle{plain} 
\newtheorem{Definition}{Definition}[section] 
\theoremstyle{plain} 
\newtheorem{Lemma}{Lemma}[section] 
\theoremstyle{plain} 
\newtheorem{Problem}{Problem}[section] 
\theoremstyle{plain} 
\newtheorem{Remark}{Remark}[section] 

\title[An explicit $PSp_4(3)$-polynomial with 3 parameters of degree 40]
{An explicit $\mathbf{PSp_4(3)}$-polynomial with 3 parameters of degree 40, with an Appendix}  
\author{Hidetaka Kitayama}
\address{Department of Mathematics, Graduate School of Science, 
Osaka University,  
Machikaneyama 1-1, Toyonaka, Osaka, 560-0043, Japan} 
\email{h-kitayama@cr.math.sci.osaka-u.ac.jp}

\begin{document} 


\begin{abstract} 
We will give an explicit polynomial over $\mathbb{Q}$ with 3 parameters 
of degree 40 as a result of the inverse Galois problem. 
Its Galois group over $\mathbb{Q} $ (resp. $\mathbb{Q}(\sqrt{-3})$) is 
isomorphic to $PGSp_4(3)$ (resp. $PSp_4(3)$)  
and it is a regular $PSp_4(3)$-polynomial over $\mathbb{Q}(\sqrt{-3})$. 
To construct the polynomial and prove its properties above  
we use some results of Siegel modular forms and permutation group theory. 
We give this polynomial in Appendix. 
\end{abstract} 

\maketitle  

\section{Introduction} 
In this paper, we will construct an explicit polynomial 
with 3 parameters of degree 40 
which has properties in Theorem \ref{main} below. 
In this section, we will explain its background. 

\begin{Definition} 
Let $G$ be a finite group and $K$ be a field. An extension $L/K$ is called 
$G$-extension over $K$ if $L/K$ is a Galois extension and the Galois group 
Gal$(L/K)$ is isomorphic to G. 
A polynomial $f(X)\in K[X]$ is called a $G$-polynomial over $K$ 
if the Galois group Gal$(f(X)/K)$ is isomorphic to $G$. 
\end{Definition} 
The inverse Galois problem asks whether there exists a $G$-extension over 
$K$ for a given field $K$ and a finite group $G$. 
This problem has been studied by many mathematicians as one of the most 
important problems in number theory, especially for the case of the rational 
number field $\mathbb{Q}$. 
It is still unknown whether this problem is affirmative for every finite 
group, but it has been proven affirmatively for a lot of kinds of finite 
groups. (See \cite{M99}. )
In this paper, we will consider the constructive aspects of this problem. 
Our problem is formulated as follows: 
\begin{Problem} \label{Prob} 
Construct an explicit polynomial having $G$ as the Galois group 
for a given transitive permutation group $G$. 
\end{Problem} 
Note that Problem \ref{Prob} for groups which are not conjugate as 
subgroups of the symmetric group are distinct from each other 
even if they are isomorphic as abstract groups. 
Complete results of Problem \ref{Prob} for all transitive permutation groups of 
degree up to 15 were given in \cite{KM}. 
The purpose of this paper is to give results for transitive permutation 
groups of degree 40. 
Our main theorem is as follows. We will prove this theorem in section 5. 
\begin{Theorem} \label{main} 
A polynomial $F(x,y,z;X)\in \mathbb{Q}(x,y,z)[X]$ 
with $3$ parameters $x,y,z$ of degree $40$, given in Appendix,  
has the following properties$:$ \\ 
$(1)$ the Galois group of $F(x,y,z;X)$ over $\mathbb{Q}(x,y,z)$ is conjugate to 
the primitive group $(40,4)$ in the GAP code. 
It is isomorphic to $PGSp_4(3)$, \\ 
$(2)$ the Galois group of $F(x,y,z;X)$ over $\mathbb{Q}(\sqrt{-3})(x,y,z)$ is 
conjugate to the primitive group $(40,3)$ in the GAP code. 
It is isomorphic to $PSp_4(3)$, \\ 
(3) $F(x,y,z;X)$ is a regular $PSp_4(3)$-polynomial over 
$\mathbb{Q}(\sqrt{-3})$. 
\end{Theorem} 
Here a regular $G$-polynomial is defined as follows:
\begin{Definition} 
A polynomial $f(\mathbf{t};X)\in K(\mathbf{t})[X]$ 
with some parameters $\mathbf{t}=(t_1,\cdots ,t_n)$ is called regular 
if it satisfies $\mathrm{Spl}(f(\mathbf{t};X)/K(\mathbf{t})) \cap \overline{K} = K$, 
where $\mathrm{Spl}(f(\mathbf{t};X)/K(\mathbf{t}))$ is the splitting field of 
$f(\mathbf{t};X)$ over $K(\mathbf{t})$ and 
$\overline{K}$ is an algebraic closure of $K$. 
\end{Definition} 
If $K$ is a Hilbertian field, especially algebraic number field,  
it is known that we can obtain infinitely many $G$-extensions over $K$ by 
specialization of a regular $G$-polynomial over $K$. (See \cite{JLY02}. ) 

We give some remarks concerning Theorem \ref{main}. 
\begin{Remark} 
Explicit polynomials over $\mathbb{Q}$ with $1$ parameter for $PSp_4(3)$ and 
$PGSp_4(3)$ as transitive permutation groups of degree $27$ are given in $p.412$ of 
\cite{M99} and 
they are regular $PSp_4(3)$- and $PGSp_4(3)$-polynomials over $\mathbb{Q}$. 
As mentioned in Theorem \ref{main} our polynomial $F(x,y,z;X)$ is just a regular 
$PSp_4(3)$-polynomial over $\mathbb{Q}(\sqrt{-3})$.
But our result is new because explicit polynomials with $3$ parameters 
for these two groups 
as transitive permutation groups of degree $40$ have not been known before. 
\end{Remark} 

\begin{Remark} 
We will construct our polynomial 
by using some results of Siegel modular forms in section $4$. 
Note that we only have its Galois group over $\mathbb{C}(x,y,z)$ at this 
stage and 
it is a difficult problem to descent fields of definition to $\mathbb{Q}$.  
In the case of $SL_2(\mathbb{Z})$, Shih studied this problem in \cite{S74} 
by using the theory of canonical systems of models  
and achieved regular Galois extensions over $\mathbb{Q}$. 
In our case, so far, we can not improve our polynomial to have regularity over $\mathbb{Q}$. 
\end{Remark} 

\ \\ 
\textbf{Notation.} 
The groups $PSp_4(3)$ and $PGSp_4(3)$ are defined as follows. 
We define the symplectic group by 
\[ Sp_4(\mathbb{Z}) := \{ g\in M_4(\mathbb{Z}) | {}^tg\cdot J\cdot g = J\} \] 
where J:=$\left( \begin{array}{cc} 0_2 & 1_2 \\ -1_2 & 0_2 \end{array} \right) $ 
 and $1_2$ is the unit matrix. 
The following two subgroups of $Sp_4(\mathbb{Z})$ will be important. 
\begin{eqnarray*} 
\Gamma _0(3) &:=& \{ g\in Sp_4(\mathbb{Z}) | g\equiv \left( \begin{array}{cc} A & B \\                    0_2 & D \end{array} \right) \ \mathrm{mod}\  3 \} ,\\ 
\Gamma (3)   &:=& \{ g\in Sp_4(\mathbb{Z}) | g \equiv 1_4 \ \mathrm{mod}\  3 \} .
\end{eqnarray*} 
By definition, we have $\Gamma (3) < \Gamma _0(3) < Sp_4(\mathbb{Z})$. 
It is known that the index $(Sp_4(\mathbb{Z}):\Gamma _0(3))=40$ and 
$(Sp_4(\mathbb{Z}):\Gamma (3))=51840$. 
$\Gamma (3)$ is a normal subgroup of $Sp_4(\mathbb{Z})$, so we put 
$Sp_4(3):=Sp_4(\mathbb{Z})/\Gamma (3)$ and define 
\[ PSp_4(3):=Sp_4(3)/\{ \pm 1_4\}. \] 
This is a non-abelian simple group of order 25920. 
We also define 
\[ PGSp_4(3):=\{ g\in GL(4,\mathbb{F}_3)\ |\ {}^tg\cdot J\cdot g=v\cdot J\ ,\ v\in \mathbb{F}_3^{\times } \} /\{ \pm 1_4\} . \] 
This is a non-abelian group of order 51840 which is isomorphic to 
$PSp_4(3)\rtimes C_2$.

\ \\ 
Acknowledgements: The author would like to thank to Professor Tomoyoshi 
Ibukiyama of Osaka University for giving him this problem and 
various advices. 

\section{Preliminaries from Siegel modular forms} 

We will consider a $PSp_4(3)$-extension by using some results of Siegel 
modular forms. So we review Siegel modular forms to fix notation. 
We denote by $H_2$ the Siegel upper half plane of degree 2, that is,  
\[ H_2 :=\{ Z\in M_2(\mathbb{C}) | {}^tZ = Z , Im(Z)>0\} . \] 
The group $Sp_4(\mathbb{Z})$ acts on $H_2$ by 
\[ gZ = (AZ+B)(CZ+D)^{-1} \] 
for any $g =\left( \begin{array}{cc} A & B \\ C & D \end{array} \right) 
\in Sp_4(\mathbb{Z})$ and any $Z \in H_2$. 
For any natural number $k$ and a holomorphic function $F$ on $H_2$, 
we put 
\[ F|[g]_k(Z) = \mathrm{det}(CZ+D)^{-k}F(gZ). \] 
For a finite index subgroup $\Gamma $ of $Sp_4(\mathbb{Z})$, we denote by 
$A_k(\Gamma )$ the space of all Siegel modular forms of weight $k$ of 
$\Gamma $, that is,  
\[ A_k(\Gamma )=\{ F:\mathrm{a\ holomorphic\ function\ on\ } H_2 \ | \ 
 F|[g]_k=F \   \mathrm{for\ all\ } g \in \Gamma \} . \] 
We put $A(\Gamma )=\oplus ^{\infty }_{k=0} A_k(\Gamma )$. 
The space $A(\Gamma )$ is a graded ring. 
The explicit structure of $A(Sp_4(\mathbb{Z}))$ and $A(\Gamma _0(3))$ is 
known as in the following theorems. 

\begin{Theorem}[Igusa\cite{Ig64}] 
\[ \bigoplus _{k=0}^{\infty } A_k(Sp_4(\mathbb{Z})) =
\mathbb{C}[\phi _4,\phi _6, \chi _{10}, \chi _{12}] \oplus 
\chi _{35}\mathbb{C}[\phi _4,\phi _6, \chi _{10}, \chi _{12}]. \] 
$\phi _4,\phi _6, \chi _{10}, \chi _{12}$ are algebraically independent over $\mathbb{C}$. 
\end{Theorem} 

\begin{Theorem}[Ibukiyama\cite{Ib91}, Aoki and Ibukiyama\cite{AIb05}]  

We put 
\[ B:=\mathbb{C}[\alpha _1,\beta _3,\gamma _4,\delta _3]\ ,\ 
   C:=\mathbb{C}[\alpha _1^2,\beta _3^2,\gamma _4,\delta _3^2] \] 
then 
\[ \bigoplus _{k=0}^{\infty } A_k(\Gamma _0(3))=
   B^{\left( even\right) }\oplus C\alpha _1\chi _{14}\oplus C\beta _3\chi _{14}\oplus 
    C\delta _3\chi_{14}\oplus C\alpha _1\beta _3\delta _3\chi _{14}. \] 
$\alpha _1,\beta _3,\gamma _4,\delta _3$ are algebraically independent over $\mathbb{C}$.
\end{Theorem} 

\section{Preliminaries from permutation group theory} 

In this section, we review permutation group theory. 
A subgroup $G$ of $S_n$, the symmetric group of degree $n$, is called 
transitive if arbitrary two elements in $\{ 1,\cdots , n\} $ can be permuted 
each other by $G$-action. 
It is well known that the Galois group of an irreducible separable 
polynomial of degree $n$ is a transitive subgroup of $S_n$. 
GAP\cite{GAP} has data bases of classification of transitive subgroups of 
$S_n$ for $n$ up to 30. 
These are based on Hulpke\cite{H05}. 
By using the data, we see the following fact.    
\begin{Lemma} 
The least degree of which $PSp_4(3)$ can be realized as a transitive subgroup 
of the symmetric group is 27.
\end{Lemma} 
Next we consider primitive groups, which are of special type in transitive 
groups.  
\begin{Definition} 
Let $G$ be a transitive subgroup of $S_n$. 
$G$ is called primitive if there are no partitions of $\{ 1,\cdots ,n\} $
which satisfy the following two conditions. 
\begin{enumerate} 
\item[(1)] $\{ 1,\cdots ,n\} =\cup ^{r}_{i=1}B_i,\ r\ge 2, \\ 
 \sharp B_i\ge 2\ (i=1,\ldots ,r)$ and $B_i\cap B_j=\emptyset \ (i\neq j).$
\item[(2)] $G$ induces a transitive action on $\{B_1,\ldots ,B_r\} $. 
\end{enumerate} 
\end{Definition} 
GAP\cite{GAP} has data bases of classification of primitive subgroups of $S_n$
for $n$ up to 2499. 
These are based on Colva M. Roney-Dougal \cite{RC05}. 
Among these, we will use the following. 
\begin{Lemma} 
Primitive groups of degree 40 are one of the following 8 groups up to 
conjugacy. 
\[ PSp_4(3)a, PSp_4(3)b, PGSp_4(3)a, PGSp_4(3)b, \] 
\[ PSL_4(3), PGL_4(3), A_{40}, S_{40}. \] 

(The symbols $^^ ^^ a"$ and $^^ ^^ b"$ mean that one group is isomorphic 
but is not conjugate to the other.)
\end{Lemma} 

\section{Construction of a $PSp_4(3)$-polynomial} 

In this section, we will consider $PSp_4(3)$-extension by using Theorem 2.1 
and 2.2 and construct a polynomial which has $PSp_4(3)$ as its Galois group 
over the rational function field over $\mathbb{C}$ of dimension 3. 

For a finite index subgroup $\Gamma $ of $Sp_4(\mathbb{Z})$, we denote by 
$K(\Gamma )$ the modular function field of $\Gamma$, 
that is, the field which consists of meromorphic functions on $H_2$ 
which are invariant  with respect to the action of $\Gamma$. 
It is known that $K(\Gamma)$ is generated by fractions of modular forms for $\Gamma$ 
of the same weight. (c.f. corollary (i) of p.131 of \cite{Kl}).  
We consider a sequence of the modular function fields of $Sp_4(\mathbb{Z})$, 
$\Gamma _0(3)$ and $\Gamma (3)$:   
\[ K(Sp_4(\mathbb{Z}))\subset K(\Gamma _0(3)) \subset K(\Gamma (3)). \] 
It is known that these three fields are purely transcendental over 
$\mathbb{C}$ of dimension 3. (c.f. \cite{Ig64},\cite{Ib91},\cite{Geer}).  
Note that $K(\Gamma (3))/K(Sp_4(\mathbb{Z}))$ 
is a $PSp_4(3)$-extension and $K(\Gamma _0(3))/K(Sp_4(\mathbb{Z}))$ is 
a non-Galois extension of degree 40. 
We will compute a polynomial which defines 
the extension $K(\Gamma _0(3))/K(Sp_4(\mathbb{Z}))$. 
Then the splitting field of the polynomial over $K(Sp_4(\mathbb{Z}))$ is 
$K(\Gamma (3))$ because the Galois group 
Gal$(K(\Gamma (3))/K(Sp_4(\mathbb{Z})))\simeq PSp_4(3)$ is simple. 
Thus the polynomial is a $PSp_4(3)$-polynomial over 
$\mathbb{C}$ with 3 parameters of degree 40.  
(We will consider the Galois group over $\mathbb{Q}$ in section 5.) 
To carry out this computation, we need the following three steps. 
\begin{enumerate} 
\item[1.] We will compute transcendental basis over $\mathbb{C}$ of 
$K(Sp_4(\mathbb{Z}))$ and $K(\Gamma _0(3))$. 
\item[2.] We will find a primitive element of the extension 
$K(\Gamma _0(3))/K(Sp_4(\mathbb{Z}))$.  
\item[3.] We will compute the irreducible polynomial of the element of 
Step2 over $K(Sp_4(\mathbb{Z}))$. This is a polynomial we want. 
\end{enumerate} 

\ \\ 

\textbf{Step 1.} We will compute transcendental basis over $\mathbb{C}$ of 
$K(Sp_4(\mathbb{Z}))$ and $K(\Gamma _0(3))$ by using Theorem 2.1 and 2.2. 
By the structure of $\oplus _{k=0}^{\infty }A_k(Sp_4(\mathbb{Z}))$, 
we see that the former (resp. latter) part of it consists of 
even (resp. odd) weight functions. 
We also see that every odd weight function is a 
product of $\chi _{35}$ and a even weight function. 
So we see that $K(Sp_4(\mathbb{Z}))$ consists of fractions of functions 
of the same weight 
belonging to $\mathbb{C}[\phi_4,\phi_6,\chi_{10},\chi_{12}]$. 
Hence $K(Sp_4(\mathbb{Z}))$ is generated by all functions of the form 
\[ {\phi _4}^a{\phi _6}^b{\chi _{10}}^c{\chi _{12}}^d,\  
(a, b, c, d\in \mathbb{Z}, 2a+3b+5c+6d=0). \] 
We can determine transcendental basis of $K(Sp_4(\mathbb{Z}))$ over 
$\mathbb{C}$ by computing  $\mathbb{Z}$-basis of the free $\mathbb{Z}$ module 
of rank 3, 
\[ V:=\{ \mathbf{x} \in \mathbb{Z}^4 | (2\ 3\ 5\ 6)\mathbf{x} =0\} \] 
because $\phi _4, \phi _6, \chi _{10}, \chi_{12}$ are algebraically  
independent over $\mathbb{C}$. 

We put 
\[ A:=\left( \begin{array}{cccc} 1&0&0&0 \\ 0&1&0&0 \\ 2&3&6&-1 \\ 
-2&-3&-5&1 \end{array} \right) .\] 
Then we have 
\begin{eqnarray*} 
V &=& \{ \mathbf{x} \in \mathbb{Z}^4 | (0\ 0\ 0\ 1)A^{-1}\mathbf{x}=0 \} \\ 
  &=& \mathbb{Z}A\left( \begin{array}{c} 1 \\ 0 \\ 0 \\ 0 \end{array} \right) 
      \oplus 
      \mathbb{Z}A\left( \begin{array}{c} 0 \\ 1 \\ 0 \\ 0 \end{array} \right) 
      \oplus 
      \mathbb{Z}A\left( \begin{array}{c} 0 \\ 0 \\ 1 \\ 0 \end{array} \right) \\
  &=& \mathbb{Z}\left( \begin{array}{c} 1 \\ 0 \\ 2 \\ -2 \end{array} \right) 
      \oplus 
      \mathbb{Z}\left( \begin{array}{c} 0 \\ 1 \\ 3 \\ -3 \end{array} \right) 
      \oplus 
      \mathbb{Z}\left( \begin{array}{c} 0 \\ 0 \\ 6 \\ -5 \end{array} \right) .
\end{eqnarray*} 
It follows that 
\[ K(Sp_4(\mathbb{Z}))= 
\mathbb{C}\left( \frac{\phi _4{\chi _{10}}^2}{{\chi _{12}}^2} , 
                 \frac{\phi _6{\chi _{10}}^3}{{\chi _{12}}^3} , 
                 \frac{{\chi _{10}}^6}{{\chi _{12}}^5} 
          \right). \] 
          
Next we will determine transcendental basis of $K(\Gamma _0(3))$. 
By the structure of $A(\Gamma _0(3))$, we see that the first part of it 
consists of even weight functions and the other parts odd weight. 
We also see that every odd weight function is a product of $\chi _{14}$ 
and $\alpha_1, \beta _3, \gamma _4, \delta _3$. Thus we can see 
$K(\Gamma _0(3))$ is generated by all functions of the form 
\[ {\alpha _1}^a{\beta _3}^b{\gamma _4}^c{\delta _3}^d,\  
(a, b, c, d\in \mathbb{Z}, a+3b+4c+3d=0). \] 
We can determine transcendental basis by the same way as above 
and get 
\[ K(\Gamma _0(3))=\mathbb{C}\left( \frac{\beta _3}{\alpha _1^3}\ ,\ 
                              \frac{\gamma _4}{\alpha _1^4}\ ,\ 
                              \frac{\delta _3}{\alpha _1^3} \right) . \] 
For simplicity, we put 
\[ a:=\frac{\beta _3}{{\alpha _1}^3}\ ,\   
b:=\frac{\gamma _4}{{\alpha _1}^4}\ ,\   
c:=\frac{\delta _3}{{\alpha _1}^3}\ , \] 
\[ x:=\frac{\phi _4{\chi _{10}}^2}{{\chi _{12}}^2}\ ,\  
y:=\frac{\phi _6{\chi _{10}}^3}{{\chi _{12}}^3}\ ,\  
z:=\frac{{\chi _{10}}^6}{{\chi _{12}}^5}\ , \] 
then we have 
\[ K(Sp_4(\mathbb{Z}))=\mathbb{C}(x,y,z) \ ,\     
   K(\Gamma _0(3))=\mathbb{C}(a,b,c). \] 

\ \\ 

\textbf{Step 2.}  We will find a primitive element of the extension 
$K(\Gamma _0(3))/K(Sp_4(\mathbb{Z}))$.  
We have the following relations (\cite{AIb05},p.259). 

\begin{eqnarray*} 
\phi _4 &=& 8\alpha _1\beta _3+41\alpha _1^4-162\gamma _4 +5\alpha _1\delta _3 \\ 
\phi _6 &=& 277\alpha _1^6-2187\alpha _1^2\gamma _4+80\alpha _1^3\beta _3 
         +\frac{11}{8}\delta _3^2+\beta _3^2+\frac{7}{2}\beta _3\delta _3 
         +\frac{91}{2}\alpha _1^3\delta _3 \\ 
\chi_{10} &=& \gamma _4(8\alpha _1^3+2\beta _3-\delta _3 )^2/6144 \\ 
\chi _{12} &=& (-124416\alpha _1^8\gamma _4-192\alpha _1^3\beta _3^2\delta _3-768\alpha _1^6\beta _3\delta _3 
            +16\alpha _1^3\delta _3^3+256\alpha _1^3\beta _3^3 \\ 
           & & +4096\alpha _1^9\beta _3+ 153\alpha _1^6\beta _3^2-1024\alpha _1^9\delta _3+16\beta _3^4-\delta _3^4  
            -16\beta _3^3\delta _3+4\beta _3\delta _3^3 \\ 
           & & +4096\alpha _1^{12}+7776\alpha _1^2\beta _3\delta _3\gamma _4+ 10077696\alpha _1^4\gamma _4^2-62208\alpha _1^5\beta _3\gamma _4  \\ 
           & & +31104\alpha _1^5\delta _3\gamma _4+2519424\alpha _1\beta _3\gamma _4^2 -1944\alpha _1^2\delta _3\gamma _4-7776\alpha _1^2\beta _3^2\gamma _4 \\ 
           & & -68024448\gamma _4^3-1259712\alpha _1\delta _3\gamma _4^2 )/3981312 \\ 
\end{eqnarray*}

By definition of $a$, $b$ and $c$, it is easy to see that  
\begin{eqnarray*}
\phi _4/\alpha _1^4 &=& 8a-162b+5c+41 \\ 
\phi _6/\alpha _1^6 &=& 80a+\frac{91}{2}c-2187b+a^2+\frac{7}{2}ac +\frac{11}{8}c^2+277 \\ 
\chi _{10}/\alpha _1^{10}  &=& b(8+2a-c)^2/6144 \\ 
\chi _{12}/\alpha _1^{12}  &=& (16a^4 + (-16c + 256)a^3 + (-7776b -192c + 1536)a^2 \\ 
                           & & + (2519424b^2 +(7776c - 62208)b + (4c^3 - 768c + 4096))a \\ 
                           & & + (-68024448b^3 + (-1259712c + 10077696)b^2  \\ 
                           & & +(-1944c^2 + 31104c - 124416)b +(-c^4 + 16c^3 - 1024c)) \\ 
                           & & /3981312. 
\end{eqnarray*}

We put 
\[ \theta :=\frac{\chi _{10}{\alpha _1}^2}{\chi _{12}}. \] 
We will prove that $\theta $ is a primitive element of the extension 
$K(\Gamma _0(3))/K(Sp_4(\mathbb{Z}))$. We have 
\begin{eqnarray*} 
K(Sp_4(\mathbb{Z}))(\theta ) &=& \mathbb{C}(x,y,z,\theta ) \\ 
                             &=& \mathbb{C}(x/\theta ^2,y/\theta ^3,
                                            z/\theta ^5,z/\theta ^6) \\ 
                             &=& \mathbb{C}(\phi _4/\alpha _1^4,\phi _6/\alpha _1^6,
                                            \chi _{10}/\alpha _1^{10},\chi _{12}/\alpha _1^{12}) \\ 
                             &=& \mathbb{C}(s,t,u,v), 
\end{eqnarray*}
where 
\[ s:=\phi _4/\alpha _1^4-41,t:=\phi _6/\alpha _1^6-277 \] 
\[ u:=6144\chi _{10}/\alpha _1^{10},v:=3981312\chi _{12}/\alpha _1^{12} . \] 
Note that 
\begin{eqnarray*}
s &=& 8a-162b+5c \\ 
t &=& 80a+\frac{91}{2}c-2187b+a^2+\frac{7}{2}ac +\frac{11}{8}c^2 \\ 
u &=& b(8+2a-c)^2 \\ 
v &=& 16a^4 + (-16c + 256)a^3 + (-7776b -192c + 1536)a^2 \\ 
  & & + (2519424b^2 + (7776c - 62208)b + (4c^3 - 768c + 4096))a \\ 
  & & + (-68024448b^3 + (-1259712c + 10077696)b^2 \\ 
  & & + (-1944c^2 + 31104c - 124416)b +(-c^4 + 16c^3 - 1024c)). 
\end{eqnarray*}   
We have to show that $a$, $b$ and $c$ belong to $\mathbb{C}(s,t,u,v)$ 
in order to show $\mathbb{C}(s,t,u,v)=\mathbb{C}(a,b,c)$. 
We define 
\begin{eqnarray*}
f &:=& 8a-162b+5c-s \\ 
g &:=& 80a+\frac{91}{2}c-2187b+a^2+\frac{7}{2}ac +\frac{11}{8}c^2-t \\ 
h &:=& b(8+2a-c)^2-u \\ 
j &:=& 16a^4 + (-16c + 256)a^3 + (-7776b -192c + 1536)a^2 \\ 
  & & + (2519424b^2 + (7776c - 62208)b + (4c^3 - 768c + 4096))a \\ 
  & & + (-68024448b^3 + (-1259712c + 10077696)b^2 \\ 
  & & + (-1944c^2 + 31104c - 124416)b +(-c^4 + 16c^3 - 1024c))-v. 
\end{eqnarray*}
Here we assume $a$, $b$, $c$, $s$, $t$, $u$ and $v$ are seven independent 
variables. 
We compute a Gr\"obner basis of the ideal $<f,g,h,j>$ relative to 
the lexicographic order with $a>c>v>b>u>t>s$. 
Then we get a Gr\"obner basis which consists of 14 polynomials.  
Among them, one finds two polynomials without $a$, $c$ 
from which a linear equation of $b$ over $\mathbb{Z}[s,t,u,v]$ can be obtained. 
This implies $b\in \mathbb{C}(s,t,u,v)$. 
Also one finds another polynomial which is linear in $c$ over $\mathbb{Z}[u,v,s,t,b]$. 
Thus we see that $c\in \mathbb{C}(u,v,s,t,b)$. 
Finally, by definition of $s$: $s=8a-162b+5c$, 
we have $a \in \mathbb{C}(s,t,u,v)$. 
It follows that 
\[ K(Sp_4(\mathbb{Z}))(\theta )=K(\Gamma _0(3)). \]      
Thus Step2 has been completed.
\ \\ 

\textbf{Step 3.} We will compute the irreducible polynomial of the element 
$\theta $ of Step2 over $K(Sp_4(\mathbb{Z}))$. 
We can get a polynomial relation of $s$, $t$, $u$ and $v$ by eliminating 
$b$ from $G_1=0$ and $G_2=0$. Then we replace $s$, $t$, $u$ and $v$ 
in this relation by 
\[ s=x/\theta ^2-41\ ,\ t=y/\theta ^3-277 \] 
\[ u=6144z/\theta ^5\ ,\ v=3981312z/\theta ^6-4096 \] 
and multiply it by $\theta ^{40}$.  
Then we get a polynomial of $\theta $ over $\mathbb{Q}[x,y,z]$ which has degree 40. 
This is the irreducible polynomial of $\theta $ over 
$K(Sp_4(\mathbb{Z}))=\mathbb{C}(x,y,z)$ and its Galois group over 
$\mathbb{C}(x,y,z)$ is isomorphic to $PSp_4(3)$. 
We omit this polynomial here because it takes up too much space. 
We give an explicit form of $F(x,y,z;X)$ in Appendix.  
     
\section{Proof of Theorem \ref{main}} 

The polynomial $F(x,y,z;X)$ obtained in the previous section has $PSp_4(3)$ 
as Galois group over $\mathbb{C}(x,y,z)$. Actually, the coefficients of 
$F(x,y,z;X)$ are in $\mathbb{Q}[x,y,z]$. 
What can be said about its Galois group over an intermediate field of 
$\mathbb{C}/\mathbb{Q}$ ? 
By answering this question, we can prove Theorem \ref{main}. 
For simplicity, we denote $\mathbb{Q}(x,y,z)$ and $\mathbb{C}(x,y,z)$ 
by $K$ and $K'$ respectively. We also denote by $L$ and $L'$ the splitting 
field of $F(x,y,z;X)$ over $K$ and $K'$ respectively. 
\begin{Lemma} 
$PSp_4(3)$ is primitive as a transitive group of degree 40. 
\end{Lemma} 
\begin{proof} 
Suppose $G\simeq PSp_4(3)$ is not primitive. 
Then $G$ has a partition of $\{1,\cdots ,40\} $ satisfying the condition of 
Definition 3.1.  
The group $H:=\{g\in G\ |\ gB_i=B_i\ ,\ \forall i=1,\ldots ,r\} $ 
is a normal subgroup of $G$. 
Because $PSp_4(3)$ is simple, $H$ is 1 or $G$. 
If $H=G$, then it contradicts transitivity of $G$. 
If $H=1$, then $G\simeq PSp_4(3)$ as a permutation group of 
$\{ B_1,\ldots ,B_r\} $. 
This implies $r\mid 40$ and it contradicts Lemma 3.1. 
\end{proof} 

By definition of the symbols $L'$ and $K'$, Gal$(L'/K')$ is a transitive 
group of degree 40. So Gal$(L'/K')$ is primitive by Lemma 5.1. 
Because 
\[ \mathrm{Gal}(L/(L\cap K')) \simeq \mathrm{Gal}(L'/K') \] 
and Gal$(L/K)$ has this group as a subgroup, we see that Gal$(L/K)$ is also 
a primitive group of degree 40. 
Thus Gal$(L/K)$ is conjugate to one of the following 
8 groups by Lemma 3.2. 
\[ PSp_4(3)a\ ,\ PSp_4(3)b\ ,\ PGSp_4(3)a\ ,\ PGSp_4(3)b, \] 
\[ PSL_4(3)\ ,\ PGL_4(3)\ ,\ A_{40}\ ,\ S_{40}. \] 
We have to determine Gal$(L/K)$ out of them. Because the intermediate field 
$L\cap K'$ is a Galois extension over $K$, Gal$(L/(L\cap K'))$ is a 
normal subgroup of Gal$(L/K)$. 
So we can exclude groups not containing $PSp_4(3)$ as their normal subgroups. 
Thus Gal$(L/K)$ is conjugate to one of 
\[ PSp_4(3)a\ ,\ PSp_4(3)b\ ,\ PGSp_4(3)a\ ,\ PGSp_4(3)b. \] 
By consulting the data base of GAP\cite{GAP}, 
we can see that $PGSp_4(3)b$ is the only group not contained in $A_{40}$ 
and the other three groups are contained in $A_{40}$. 
So the discriminant of $F(x,y,z;X)$ determines whether Gal$(L/K)$ is conjugate 
to $PGSp_4(3)b$ or not. 
We denote by $\mathbf{d}(F(x,y,z;X))$ the discriminant of $F(x,y,z;X)$. 
We can write  
\[ \mathbf{d}(F(x,y,z;X))=f(x,y,z)g(x,y,z)^2 \] 
with some $f(x,y,z)$ and $g(x,y,z)$ $\in \mathbb{Q}[x,y,z]$.  
We can assume that $f(x,y,z)$ is square free.  
On the other hand, 
\[ \mathbf{d}(F(x,y,z;X)) \in \mathbb{C}[x,y,z]^2 \] 
because Gal$(L'/K')\simeq PSp_4(3)$, which is contained in $A_{40}$. 
Thus $f(x,y,z)$ must be constant and 
\[ \mathbf{d}(F(x,y,z;X))=r\cdot g(x,y,z)^2 \] 
for a square free integer $r$. 
Whether Gal$(L/K)$ is contained in $A_{40}$ or not depends on $r$.    
By calculation of the discriminant of specialized polynomial $F(1,1,1;X)$,  
it turns out that $r\cdot g(1,1,1)^2=-3\cdot n^2$ for some $n\in \mathbb{Z}$. 
The left hand is the specialization of $\mathbf{d}(F(x,y,z;X))$ and 
the other hand is the discriminant of the specialized polynomial 
$F(1,1,1;X)$.  
Thus we get $r=-3$ and $\mathbf{d}(F(x,y,z;X))=-3\cdot g(x,y,z)^2$. 
It follows that 
if $k$ contains (resp. does not contain) $\sqrt{-3}$ then the Galois group 
of $F(x,y,z;X)$ over $k(x,y,z)$ is isomorphic to 
$PSp_4(3)$ (resp.$PGSp_4(3)$) for an intermediate field $k$ of 
$\mathbb{C}/\mathbb{Q}$ . 
Thus we can see that $F(x,y,z;X)$ is a regular $PSp_4(3)$-polynomial 
over $\mathbb{Q}(\sqrt{-3})$.

\vspace{10mm} 
\section*{Appendix} 

In this appendix, we give the explicit form of $F(x,y,z;X)$ which was 
constructed in section 4. 
It is a 3-parameter polynomial over $\mathbb{Q}$ of degree 40. 
(See Theorem \ref{main}. ) 
We denote 3 parameters by small letters $x,y,z$ and main variable by 
capital letter $X$.

\ 

$F(x,y,z;X)=$ 

\ 

$19683 X^{40}$ 

\ 

$+(-708588 x)\ X^{38}$ 

\ 

$+(-118098 y)\ X^{37}$ 

\ 

$+(8621154 x^2)\ X^{36}$ 

\ 

$+(1180980 y x
 + 5904900 z)\ X^{35}$ 

\ 

$+(-52671708 x^3
 + (-452709 y^2
 + 36610380 z) )\ X^{34}$  

\ 

$+(-3109914 y x^2
 - 66213612 z x)\ X^{33}$ 

\ 

$+(179292447 x^4
 + (5983632 y^2
 - 346027140 z)  x
 - 9447840 z y)\ X^{32}$ 

\ 

$+(-4408992 y x^3
 + 297501984 z x^2
 + (-647352 y^3
 + 63615456 z y) )\ X^{31}$ 

\ 

$+(-319022064 x^5
 + (-33959736 y^2
 + 1033593696 z)  x^2
 + 151375392 z y x
 - 161151282 z^2)\ X^{30}$ 

\ 

$+(16953624 y x^4
 - 654840288 z x^3
 + (8083152 y^3
 - 449087328 z y)  x\\ 
+ (-31597776 z y^2
 + 1280733444 z^2) )\ X^{29}$  

\ 

$+(123469272 x^6
 + (112429296 y^2
 - 96787872 z)  x^3\\ 
 - 917000352 z y x^2
 + 1442702664 z^2 x\\ 
 + (-520506 y^4
 + 43355088 z y^2
 - 2276653878 z^2) )\ X^{28}$  

\ 

$+(121247280 y x^5
 + 537792048 z x^4
 + (-42742728 y^3
 + 673316064 z y)  x^2\\ 
 + (351354672 z y^2
 - 6071969304 z^2)  x
 - 556884558 z^2 y)\ X^{27}$ 

\ 

$+(619883280 x^7
 + (-261049068 y^2
 - 5103408240 z)  x^4
 + 2739663648 z y x^3\\ 
 - 5087701206 z^2 x^2
 + (5756184 y^4
 - 196830000 z y^2
 + 12097250532 z^2)  x\\ 
 + (-28903392 z y^3
 + 1308919500 z^2 y) )\ X^{26}$  

\ 

$+(-670289256 y x^6
 + 953077104 z x^5
 + (125901216 y^3
 + 1820073888 z y)  x^3\\ 
 + (-1497727584 z y^2
 + 2692398204 z^2)  x^2
 + 4143516444 z^2 y x\\ 
 + (-265356 y^5
 + 14206752 z y^3
 - 1402335018 z^2 y
 - 4631937696 z^3) )\ X^{25}$  

\ 

$+(-1156831146 x^8
 + (487193616 y^2
 + 7684768080 z)  x^5\\ 
 - 4363852320 z y x^4
 + 8687297628 z^2 x^3
 + (-26106948 y^4
 - 69914016 z y^2
 - 12500397738 z^2)  x^2\\ 
 + (247440096 z y^3
 - 4506934608 z^2 y)  x\\ 
 + (-461767554 z^2 y^2
 + 22306087800 z^3) )\ X^{24}$  

\ 

$+(1285396128 y x^7
 - 3945348000 z x^6
 + (-230941368 y^3
 - 5241451680 z y)  x^4\\ 
 + (3003923232 z y^2
 + 22965074640 z^2)  x^3
 - 11301689916 z^2 y x^2\\ 
 + (2490264 y^5
 - 28203552 z y^3
 + 1402216920 z^2 y
 + 9935742204 z^3)  x\\ 
 + (-12859560 z y^4
 + 513148932 z^2 y^2
 - 43360468020 z^3) )\ X^{23}$  

\ 

$+(534555288 x^9
 + (-699845832 y^2
 + 2532511008 z)  x^6\\ 
 + 4179934368 z y x^5
 - 7212797442 z^2 x^4
 + (62256600 y^4
 + 1129051872 z y^2
 - 26452639800 z^2)  x^3\\
 + (-758587680 z y^3
 + 1762704504 z^2 y)  x^2
 + (2563700544 z^2 y^2
 - 17495886276 z^3)  x\\
 + (-89586 y^6
 + 2245320 z y^4
 + 115644186 z^2 y^2
 - 1238371254 z^3 y
 + 35599854780 z^3) )\ X^{22}$  

\ 

$+(-915737724 y x^8
 + 6241336416 z x^7
 + (278664624 y^3
 - 701309664 z y)  x^5\\ 
 + (-2786937840 z y^2
 - 28199606652 z^2)  x^4
 + 12942899280 z^2 y x^3\\ 
 + (-9137772 y^5
 - 75232800 z y^3
 + 16261805916 z^2 y
 - 17637096492 z^3)  x^2\\ 
 + (80258040 z y^4
 - 1555796808 z^2 y^2
 + 6215865156 z^3)  x\\ 
 + (-151568334 z^2 y^3
 - 2236741128 z^3 y) )\ X^{21}$  

\ 

$+(421709004 x^{10}
 + (569471472 y^2
 - 11332159200 z)  x^7\\ 
 - 4050510624 z y x^6
 + 4332674448 z^2 x^5
 + (-82787670 y^4
 + 191931120 z y^2
 + 45214488522 z^2)  x^4\\ 
 + (946455840 z y^3
 + 1944750384 z^2 y)  x^3
 + (-4601373642 z^2 y^2
 + 7282631268 z^3)  x^2\\ 
 + (680400 y^6
 - 2070360 z y^4
 - 3283570548 z^2 y^2
 + 8301917610 z^3 y
 - 34297627500 z^3)  x\\ 
 + (-2993760 z y^5
 + 108017388 z^2 y^3
 + 15573950676 z^3 y
 - 1725396471 z^4) )\ X^{20}$ 

\ 

$+(-241030728 y x^9
 - 4094668584 z x^8
 + (-206655624 y^3
 + 10533688416 z y)  x^6\\ 
 + (1326348432 z y^2
 - 12787645608 z^2)  x^5
 - 4703618322 z^2 y x^4\\ 
 + (16166304 y^5
 - 118234080 z y^3
 - 29979588456 z^2 y
 + 6382691460 z^3)  x^3\\ 
 + (-155118240 z y^4
 + 2356597476 z^2 y^2
 + 99805083444 z^3)  x^2
 + (548153892 z^2 y^3
 - 10646907948 z^3 y)  x\\ 
 + (-20088 y^7
 + 272160 z y^5
 + 202304790 z^2 y^3
 - 791291106 z^3 y^2
 - 11899843164 z^3 y
 - 1741548924 z^4) )\ X^{19}$ 

\ 

$+(-309055176 x^{11}
 + (-15025662 y^2
 + 2851480584 z)  x^8
 + 4065285024 z y x^7
 - 6646027158 z^2 x^6\\ 
 + (57888432 y^4
 - 4350776976 z y^2
 + 4170880188 z^2)  x^5\\ 
 + (-383590080 z y^3
 + 9084427668 z^2 y)  x^4
 + (2461244940 z^2 y^2
 + 3681837828 z^3)  x^3\\ 
 + (-1873368 y^6
 + 29509920 z y^4
 + 7568065386 z^2 y^2
 - 6593752998 z^3 y
 - 30354414012 z^3)  x^2\\ 
 + (11671776 z y^5
 - 375277536 z^2 y^3
 - 32693839164 z^3 y
 + 6757423704 z^4)  x\\ 
 + (-20132550 z^2 y^4
 + 1897394544 z^3 y^2
 + 19171088910 z^4) )\ X^{18}$ 

\ 

$+(483199668 y x^{10}
 - 640468296 z x^9
 + (36593856 y^3
 - 4842682848 z y)  x^7\\ 
 + (-1344112704 z y^2
 + 33727432860 z^2)  x^6
 + 956864988 z^2 y x^5\\ 
 + (-13158936 y^5
 + 1090117440 z y^3
 - 2343313638 z^2 y
 + 1740557484 z^3)  x^4\\ 
 + (79820640 z y^4
 - 3548818656 z^2 y^2
 - 105014027556 z^3)  x^3
 + (-478043208 z^2 y^3
 + 13785317100 z^3 y)  x^2\\ 
 + (115344 y^7
 - 2916000 z y^5
 - 794633328 z^2 y^3
 + 974974806 z^3 y^2
 + 31937609556 z^3 y
 - 9704908728 z^4)  x\\ 
 + (-317520 z y^6
 + 13538988 z^2 y^4
 + 66292344 z^3 y^2
 - 1211736510 z^4 y
 - 11769620436 z^4) )\ X^{17}$  

\ 

$+(-169503354 x^{12}
 + (-266184144 y^2
 + 3361205160 z)  x^9\\ 
 - 783880416 z y x^8
 + 5623286004 z^2 x^7\\ 
 + (-11217528 y^4
 + 2921209920 z y^2
 - 18078857370 z^2)  x^6\\ 
 + (189376704 z y^3
 - 16059360672 z^2 y)  x^5\\ 
 + (212100444 z^2 y^2
 + 3389744052 z^3)  x^4\\ 
 + (2113776 y^6
 - 178886880 z y^4
 + 480428496 z^2 y^2
 + 341139114 z^3 y
 + 32039033868 z^3)  x^3\\ 
 + (-8582112 z y^5
 + 535902480 z^2 y^3
 + 27081906768 z^3 y
 - 1727169642 z^4)  x^2\\
 + (35943912 z^2 y^4
 - 2671147368 z^3 y^2
 - 29150350956 z^4)  x\\ 
 + (-2889 y^8
 + 99792 z y^6
 + 27432270 z^2 y^4
 - 15062652 z^3 y^3\\ 
 - 865824552 z^3 y^2
 + 2793674772 z^4 y
 - 23412725109 z^4) )\ X^{16}$ 

\ 

$+(106119072 y x^{11}
 + 1290170592 z x^{10}
 + (73542168 y^3
 - 3078059616 z y)  x^8\\ 
 + (842436288 z y^2
 - 19747921824 z^2)  x^7
 - 2961013320 z^2 y x^6\\ 
 + (2531088 y^5
 - 963088704 z y^3
 + 16439276592 z^2 y
 + 1109061720 z^3)  x^5\\ 
 + (-6318000 z y^4
 + 2892537864 z^2 y^2
 + 25402739832 z^3)  x^4
 + (-24520320 z^2 y^3
 - 6605684784 z^3 y)  x^3\\ 
 + (-208872 y^7
 + 18405792 z y^5
 - 128064888 z^2 y^3
 - 210701088 z^3 y^2
 - 17320935024 z^3 y
 + 4557019824 z^4)  x^2\\ 
 + (370224 z y^6
 - 23963688 z^2 y^4
 - 1134808056 z^3 y^2
 + 712364544 z^4 y
 + 46659522528 z^4)  x\\ 
 + (-911898 z^2 y^5
 + 33845040 z^3 y^3
 + 718315776 z^4 y
 - 568575936 z^5) )\ X^{15}$ 

\ 

$+(48736080 x^{13}
 + (6790392 y^2
 + 219399840 z)  x^{10}
 - 1318304160 z y x^9
 + 612600570 z^2 x^8\\ 
 + (-13127184 y^4
 + 1231578432 z y^2
 - 2056794768 z^2)  x^7\\ 
 + (-296488512 z y^3
 + 13826144016 z^2 y)  x^6\\ 
 + (389355984 z^2 y^2
 - 10425638952 z^3)  x^5\\ 
 + (-579852 y^6
 + 201106800 z y^4
 - 6032839500 z^2 y^2
 + 677770740 z^3 y
 - 10465949736 z^3)  x^4\\ 
 + (-2522016 z y^5
 - 85279392 z^2 y^3
 - 6912774576 z^3 y
 - 881147160 z^4)  x^3\\ 
 + (-4093578 z^2 y^4
 + 1625570856 z^3 y^2
 + 190059048 z^4)  x^2\\ 
 + (11124 y^8
 - 1051056 z y^6
 + 26241084 z^2 y^4\\ 
 + 6406452 z^3 y^3
 - 919327320 z^3 y^2
 - 1883549376 z^4 y
 + 61618340088 z^4)  x\\ 
 + (-2592 z y^7
 - 129276 z^2 y^5
 + 118488744 z^3 y^3\\ 
 - 2383830 z^4 y^2
 - 5743236960 z^4 y
 + 554150808 z^5) )\ X^{14}$ 

\ 

$+(-112161672 y x^{12}
 + 329368032 z x^{11}
 + (-19114704 y^3
 - 66523680 z y)  x^9\\ 
 + (499142736 z y^2
 - 562662612 z^2)  x^8
 + 178500672 z^2 y x^7\\ 
 + (2148552 y^5
 - 288899136 z y^3
 + 409738824 z^2 y
 - 2137284792 z^3)  x^6\\ 
 + (56434320 z y^4
 - 3858703920 z^2 y^2
 + 24562127016 z^3)  x^5\\ 
 + (10611756 z^2 y^3
 + 1592587008 z^3 y)  x^4\\ 
 + (97632 y^7
 - 27311904 z y^5
 + 1174677552 z^2 y^3
 - 218659824 z^3 y^2
 - 3899683440 z^3 y
 + 5098209984 z^4)  x^3\\ 
 + (554976 z y^6
 - 51170940 z^2 y^4
 + 2164809240 z^3 y^2
 - 81369360 z^4 y
 - 47738360880 z^4)  x^2\\ 
 + (901476 z^2 y^5
 - 124967448 z^3 y^3
 + 2585197296 z^4 y
 + 295324380 z^5)  x\\ 
 + (-242 y^9
 + 25056 z y^7
 - 1662606 z^2 y^5\\ 
 + 103356 z^3 y^4
 + 117299016 z^3 y^3
 + 4366224 z^4 y^2\\ 
 - 3559683672 z^4 y
 + 686283516 z^5) )\ X^{13}$ 

\ 

$+(58099032 x^{14}
 + (76667472 y^2
 - 100784736 z)  x^{11}
 - 173101536 z y x^{10}
 - 579404664 z^2 x^9\\ 
 + (5467770 y^4
 + 14218416 z y^2
 + 777264174 z^2)  x^8
 + (-98850240 z y^3
 - 946207008 z^2 y)  x^7\\ 
 + (-94820868 z^2 y^2
 + 4284496296 z^3)  x^6\\ 
 + (-341712 y^6
 + 41828400 z y^4
 + 366768648 z^2 y^2
 + 818500788 z^3 y
 + 14858075592 z^3)  x^5\\ 
 + (-6575904 z y^5
 + 551956680 z^2 y^3
 - 10977672744 z^3 y
 - 328024242 z^4)  x^4\\ 
 + (-4227660 z^2 y^4
 - 151326792 z^3 y^2
 + 11924013888 z^4)  x^3\\ 
 + (-9486 y^8
 + 2304288 z y^6
 - 132748470 z^2 y^4
 + 30442716 z^3 y^3\\ 
 + 2993023224 z^3 y^2
 - 1382568912 z^4 y
 - 55764169740 z^4)  x^2\\ 
 + (-48096 z y^7
 + 6908976 z^2 y^5
 - 399361752 z^3 y^3\\ 
 + 11623176 z^4 y^2
 + 9979438464 z^4 y
 + 356723028 z^5)  x\\ 
 + (-50310 z^2 y^6
 + 6399000 z^3 y^4
 - 182346228 z^4 y^2\\ 
 - 33431454 z^5 y
 - 1497263940 z^5) )\ X^{12}$ 

\ 

$+(-69627600 y x^{13}
 - 14335056 z x^{12}
 + (-24287256 y^3
 + 21827232 z y)  x^{10}\\ 
 + (33598800 z y^2
 + 1138377240 z^2)  x^9
 + 316924974 z^2 y x^8\\ 
 + (-622944 y^5
 - 5054400 z y^3
 - 751173264 z^2 y
 - 72046584 z^3)  x^7\\ 
 + (11789280 z y^4
 + 479466216 z^2 y^2
 - 13238050968 z^3)  x^6\\ 
 + (1529928 z^2 y^3
 - 127446696 z^3 y)  x^5\\ 
 + (37944 y^7
 - 3330720 z y^5
 - 153541980 z^2 y^3
 - 129858228 z^3 y^2
 + 1116192312 z^3 y
 - 1508184360 z^4)  x^4\\ 
 + (434016 z y^6
 - 34752240 z^2 y^4
 + 726276456 z^3 y^2
 + 148723776 z^4 y
 + 13667245344 z^4)  x^3\\ 
 + (-82620 z^2 y^5
 + 48866328 z^3 y^3
 - 2712439872 z^4 y
 - 35767656 z^5)  x^2\\ 
 + (468 y^9
 - 108000 z y^7
 + 8057880 z^2 y^5
 - 1975752 z^3 y^4
 - 296166456 z^3 y^3\\ 
 + 85450464 z^4 y^2
 + 5541683040 z^4 y
 + 964117080 z^5)  x\\ 
 + (1476 z y^8
 - 251316 z^2 y^6
 + 15308028 z^3 y^4\\ 
 + 1608012 z^4 y^3
 - 305392680 z^4 y^2
 - 86558544 z^5 y
 - 1738402560 z^5) )\ X^{11}$ 

\ 

$+(20441808 x^{15}
 + (31734612 y^2
 + 33241104 z)  x^{12}
 + 14004576 z y x^{11}
 - 107773362 z^2 x^{10}\\ 
 + (4044600 y^4
 + 19485360 z y^2
 + 129296412 z^2)  x^9\\ 
 + (-4848480 z y^3
 - 394679628 z^2 y)  x^8\\ 
 + (-58523256 z^2 y^2
 + 871151112 z^3)  x^7\\ 
 + (20328 y^6
 + 496800 z y^4
 + 215786916 z^2 y^2\\ 
 - 102394908 z^3 y
 - 9138458232 z^3)  x^6\\ 
 + (-788832 z y^5
 - 96577056 z^2 y^3
 + 5385402936 z^3 y
 + 249747624 z^4)  x^5\\ 
 + (2410470 z^2 y^4
 - 174578328 z^3 y^2
 - 5220465228 z^4)  x^4\\ 
 + (-2268 y^8
 + 79776 z y^6
 + 24808680 z^2 y^4
 + 7601004 z^3 y^3\\ 
 - 1014902136 z^3 y^2
 + 900730368 z^4 y
 + 20693901384 z^4)  x^3\\ 
 + (-10848 z y^7
 - 11016 z^2 y^5
 + 80830224 z^3 y^3
 - 27048168 z^4 y^2\\ 
 - 3937089888 z^4 y
 - 779458464 z^5)  x^2\\ 
 + (30480 z^2 y^6
 - 3784212 z^3 y^4
 + 91645992 z^4 y^2\\ 
 + 37340838 z^5 y
 + 2843269128 z^5)  x\\ 
 + (-9 y^{10}
 + 2124 z y^8
 - 184626 z^2 y^6
 - 18990 z^3 y^5\\ 
 + 7442604 z^3 y^4
 + 3872880 z^4 y^3 
 - 111620106 z^4 y^2
 - 165804246 z^5 y\\ 
 + (-19961586 z^6
 - 533278080 z^5) ) )\ X^{10}$ 

\ 

$+(-18256968 y x^{14}
 - 3035664 z x^{13}
 + (-7154784 y^3
 - 48255264 z y)  x^{11}\\ 
 + (1298592 z y^2
 - 141526764 z^2)  x^{10}
 + 33123924 z^2 y x^9\\ 
 + (-357084 y^5
 - 8177760 z y^3
 - 35182998 z^2 y
 - 15405768 z^3)  x^8\\ 
 + (672480 z y^4
 + 37037088 z^2 y^2
 + 3141435960 z^3)  x^7
 + (4854456 z^2 y^3
 - 347845752 z^3 y)  x^6\\ 
 + (1296 y^7
 + 99360 z y^5
 - 24708240 z^2 y^3
 + 34145388 z^3 y^2
 + 639344664 z^3 y
 - 461722032 z^4)  x^5\\ 
 + (13680 z y^6
 + 9743220 z^2 y^4
 - 530169624 z^3 y^2
 - 57998268 z^4 y
 - 602183160 z^4)  x^4\\ 
 + (-214992 z^2 y^5
 + 1819152 z^3 y^3
 + 1081567728 z^4 y
 + 18040212 z^5)  x^3\\ 
 + (54 y^9
 + 6048 z y^7
 - 1936548 z^2 y^5
 + 365724 z^3 y^4
 + 105967440 z^3 y^3\\ 
 - 60934680 z^4 y^2
 - 2338882776 z^4 y
 - 1064286540 z^5)  x^2\\ 
 + (-108 z y^8
 + 55944 z^2 y^6
 - 4687308 z^3 y^4
 - 638640 z^4 y^3\\ 
 + 93568608 z^4 y^2
 + 120417192 z^5 y
 + 2410528896 z^5)  x\\ 
 + (54 z^2 y^7
 - 61992 z^3 y^5
 + 7323696 z^4 y^3
 + 496908 z^5 y^2\\ 
 - 245288088 z^5 y
 - 67658652 z^6) )\ X^9$ 

\ 

$+(3984903 x^{16}
 + (6190128 y^2
 + 20358864 z)  x^{13}\\ 
 - 5641056 z y x^{12}
 + 10109124 z^2 x^{11}
 + (846972 y^4
 + 19769184 z y^2
 + 5825682 z^2)  x^{10}\\ 
 + (-1798560 z y^3
 + 87711984 z^2 y)  x^9
 + (-357210 z^2 y^2
 - 774512928 z^3)  x^8\\ 
 + (15120 y^6
 + 1153440 z y^4
 - 4700592 z^2 y^2
 + 36905220 z^3 y
 + 2517995160 z^3)  x^7\\ 
 + (-40608 z y^5
 - 1273968 z^2 y^3
 - 994414320 z^3 y
 - 21993768 z^4)  x^6\\ 
 + (-336960 z^2 y^4
 + 79227720 z^3 y^2
 + 541378728 z^4)  x^5\\ 
 + (-81 y^8
 - 19440 z y^6
 + 1008450 z^2 y^4
 - 4091364 z^3 y^3\\ 
 + 86086152 z^3 y^2
 - 35448840 z^4 y
 - 3298892670 z^4)  x^4\\ 
 + (864 z y^7
 - 343440 z^2 y^5
 + 6341328 z^3 y^3
 + 9607896 z^4 y^2\\ 
 + 594957312 z^4 y
 + 126349308 z^5)  x^3\\ 
 + (-918 z^2 y^6
 + 1025460 z^3 y^4
 - 51665688 z^4 y^2
 - 15725016 z^5 y
 - 917813916 z^5)  x^2\\ 
 + (-324 z y^8
 + 53460 z^2 y^6
 + 8802 z^3 y^5
 - 2695356 z^3 y^4\\ 
 - 2065824 z^4 y^3
 + 38508696 z^4 y^2
 + 140640138 z^5 y
 + (10639188 z^6
 + 688642560 z^5) )  x\\ 
 + (324 z^2 y^7
 - 82620 z^3 y^5
 - 9099 z^4 y^4
 + 6360768 z^4 y^3\\ 
 + 1339416 z^5 y^2
 - 165722112 z^5 y
 - 85571478 z^6) )\ X^8$ 

\ 

$+(-2701728 y x^{15}
 + 3425184 z x^{14}
 + (-1000104 y^3
 - 14339808 z y)  x^{12}\\ 
 + (3748896 z y^2
 - 21230640 z^2)  x^{11}
 - 12739644 z^2 y x^{10}\\ 
 + (-49896 y^5
 - 3422880 z y^3
 + 1481976 z^2 y
 + 33535980 z^3)  x^9\\ 
 + (274200 z y^4
 - 7668108 z^2 y^2
 - 567394308 z^3)  x^8
 + (-869664 z^2 y^3
 + 208449936 z^3 y)  x^7\\ 
 + (-216 y^7
 - 67104 z y^5
 + 918216 z^2 y^3
 - 6866352 z^3 y^2
 - 241153200 z^3 y
 - 10882512 z^4)  x^6\\ 
 + (-144 z y^6
 - 55656 z^2 y^4
 + 76953240 z^3 y^2
 + 554688 z^4 y
 - 69898464 z^4)  x^5\\ 
 + (35238 z^2 y^5
 - 5188896 z^3 y^3
 - 87837696 z^4 y
 + 6956496 z^5)  x^4\\ 
 + (864 z y^7
 + 39096 z^2 y^5
 + 44988 z^3 y^4
 - 11673072 z^3 y^3\\ 
 + 8476704 z^4 y^2
 + 363730176 z^4 y
 + 156247056 z^5)  x^3\\ 
 + (-2916 z^2 y^6
 + 699732 z^3 y^4
 + 61848 z^4 y^3\\ 
 - 22383216 z^4 y^2
 - 34875360 z^5 y
 - 571815936 z^5)  x^2\\ 
 + (-36 z^2 y^7
 + 10116 z^3 y^5
 - 1975536 z^4 y^3\\ 
 + 242568 z^5 y^2
 + 97355520 z^5 y
 + 26867376 z^6)  x\\ 
 + (162 z^2 y^7
 + 38 z^3 y^6
 - 29484 z^3 y^5
 - 8820 z^4 y^4
 + 1776816 z^4 y^3\\ 
 + 1205928 z^5 y^2
 + (-383508 z^6
 - 38071296 z^5)  y
 - 64820736 z^6) )\ X^7$ 

\ 

$+(477684 x^{17}
 + (661896 y^2
 + 3252960 z)  x^{14}
 - 3115296 z y x^{13}
 + 6960186 z^2 x^{12}\\ 
 + (77112 y^4
 + 3425760 z y^2
 - 460728 z^2)  x^{11}\\ 
 + (-667296 z y^3
 - 7800840 z^2 y)  x^{10}
 + (4111632 z^2 y^2
 + 103529772 z^3)  x^9\\ 
 + (1134 y^6
 + 266760 z y^4
 + 1323378 z^2 y^2
 - 21862782 z^3 y
 - 336940884 z^3)  x^8\\ 
 + (-11232 z y^5
 - 678240 z^2 y^3
 + 136309392 z^3 y
 + 22700376 z^4)  x^7\\ 
 + (67626 z^2 y^4
 - 15946632 z^3 y^2
 - 40592664 z^4)  x^6\\ 
 + (1296 z y^6
 + 17820 z^2 y^4
 + 566316 z^3 y^3
 + 2155896 z^3 y^2
 - 19956672 z^4 y
 + 220222152 z^4)  x^5\\ 
 + (8100 z^2 y^5
 - 2005560 z^3 y^3
 + 665010 z^4 y^2
 - 18996768 z^4 y
 + 49541112 z^5)  x^4\\ 
 + (-540 z^2 y^6
 + 30564 z^3 y^4
 + 6220800 z^4 y^2
 - 3754548 z^5 y
 + 44000496 z^5)  x^3\\ 
 + (-3402 z^2 y^6
 - 1158 z^3 y^5
 + 345060 z^3 y^4
 + 250560 z^4 y^3\\ 
 - 5409180 z^4 y^2
 - 26795772 z^5 y
 + (2624292 z^6
 - 176359680 z^5) )  x^2\\ 
 + (6804 z^3 y^5
 - 7728 z^4 y^4
 - 1197504 z^4 y^3\\ 
 + 356724 z^5 y^2
 + 48522240 z^5 y
 + 31477896 z^6)  x\\ 
 + (-3402 z^4 y^4
 + 9426 z^5 y^3
 + 852444 z^5 y^2\\ 
 - 868860 z^6 y
 - 41617152 z^6) )\ X^6$ 

\ 

$+(-235458 y x^{16}
 + 895392 z x^{15}
 + (-69552 y^3
 - 1416096 z y)  x^{13}\\ 
 + (744048 z y^2
 + 7262028 z^2)  x^{12}
 - 3872016 z^2 y x^{11}\\ 
 + (-2268 y^5
 - 337824 z y^3
 - 884196 z^2 y
 + 5127588 z^3)  x^{10}\\ 
 + (35640 z y^4
 + 3222072 z^2 y^2
 + 53656020 z^3)  x^9\\ 
 + (-389070 z^2 y^3
 - 41245848 z^3 y)  x^8\\ 
 + (-7776 z y^5
 - 439344 z^2 y^3
 + 3835296 z^3 y^2
 + 36201168 z^3 y
 + 63680256 z^4)  x^7\\ 
 + (64476 z^2 y^4
 - 5515128 z^3 y^2
 - 9126432 z^4 y
 - 66130992 z^4)  x^6\\ 
 + (-324 z^2 y^5
 + 436104 z^3 y^3
 - 3114288 z^4 y
 + 6262596 z^5)  x^5\\ 
 + (-2430 z^2 y^5
 - 34056 z^3 y^4
 + 215784 z^3 y^3\\ 
 + 1861056 z^4 y^2
 - 11996424 z^4 y
 + 26923428 z^5)  x^4\\ 
 + (-2916 z^3 y^4
 + 16992 z^4 y^3
 + 1314144 z^4 y^2\\ 
 - 6686064 z^5 y
 + 1306368 z^5)  x^3\\ 
 + (-324 z^3 y^5
 + 66096 z^4 y^3
 + 150228 z^5 y^2
 - 6660144 z^5 y
 + 3897720 z^6)  x^2\\ 
 + (54 z^3 y^6
 + 2916 z^3 y^5
 - 5832 z^4 y^4\\ 
 - 361584 z^4 y^3
 - 2916 z^5 y^2 
 + (-193284 z^6\\ 
 + 11197440 z^5)  y
 + 15925248 z^6)  x\\ 
 + (-54 z^4 y^5
 - 2916 z^4 y^4
 + 12960 z^5 y^3
 + 466560 z^5 y^2\\ 
 - 674568 z^6 y
 + (57204 z^7
 - 17915904 z^6) ) )\ X^5$ 

\ 

$+(35234 x^{18}
 + (37584 y^2
 + 188640 z)  x^{15}
 - 397920 z y x^{14}\\ 
 + 1151520 z^2 x^{13}
 + (2646 y^4
 + 189072 z y^2
 - 244674 z^2)  x^{12}\\ 
 + (-51552 z y^3
 - 2805264 z^2 y)  x^{11}
 + (504654 z^2 y^2
 + 7393716 z^3)  x^{10}\\ 
 + (11880 z y^4
 + 391500 z^2 y^2
 - 1463838 z^3 y
 + 12727908 z^3)  x^9\\ 
 + (-195156 z^2 y^3
 - 11162124 z^3 y
 + 1273185 z^4)  x^8\\ 
 + (2916 z^2 y^4
 + 3260520 z^3 y^2
 + 16715808 z^4)  x^7\\ 
 + (21222 z^2 y^4
 - 134652 z^3 y^3
 - 795096 z^3 y^2
 - 7841232 z^4 y
 - 19705356 z^4)  x^6\\ 
 + (-61560 z^3 y^3
 + 398736 z^4 y^2
 + 6023808 z^4 y
 + 3536460 z^5)  x^5\\ 
 + (-11772 z^3 y^4
 + 454572 z^4 y^2
 - 135054 z^5 y
 + 4012740 z^5)  x^4\\ 
 + (666 z^3 y^5
 + 972 z^3 y^4
 + 2016 z^4 y^3
 - 106920 z^4 y^2\\ 
 - 3339468 z^5 y
 + (-315936 z^6
 + 4852224 z^5) )  x^3\\ 
 + (-1974 z^4 y^4
 + 152928 z^5 y^2
 - 248832 z^5 y
 + 3115260 z^6)  x^2\\ 
 + (324 z^4 y^4
 + 1878 z^5 y^3
 - 47304 z^5 y^2
 - 254340 z^6 y
 + 2363904 z^6)  x\\ 
 + (-36 z^4 y^5
 - 729 z^4 y^4
 + 4914 z^5 y^3
 + (-531 z^6
 + 93312 z^5)  y^2\\ 
 - 176256 z^6 y
 + (106380 z^7
 - 2985984 z^6) ) )\ X^4$ 

\ 

$+(-11340 y x^{17}
 + 83268 z x^{16}
 + (-1944 y^3
 - 38880 z y)  x^{14}\\ 
 + (40176 z y^2
 + 604584 z^2)  x^{13}
 - 283662 z^2 y x^{12}\\ 
 + (-7776 z y^3
 - 52488 z^2 y
 + 538164 z^3)  x^{11}\\ 
 + (154548 z^2 y^2
 - 2283228 z^3)  x^{10}
 + (-972 z^2 y^3
 + 107892 z^3 y)  x^9\\ 
 + (-18954 z^2 y^3
 - 85914 z^3 y^2
 - 953532 z^3 y
 - 1357884 z^4)  x^8\\ 
 + (507384 z^3 y^2
 + 463104 z^4 y
 + 5388768 z^4)  x^7\\ 
 + (-44712 z^3 y^3
 - 2379456 z^4 y
 - 587304 z^5)  x^6\\ 
 + (2268 z^3 y^4
 - 17496 z^3 y^3
 + 160704 z^4 y^2
 + 1049760 z^4 y
 + 2583576 z^5)  x^5\\ 
 + (-10260 z^4 y^3
 + 17496 z^4 y^2
 - 22032 z^5 y
 - 1679616 z^5)  x^4\\ 
 + (11664 z^4 y^3
 + 15120 z^5 y^2
 - 699840 z^5 y
 - 301968 z^6)  x^3\\ 
 + (-648 z^4 y^4
 + 27216 z^5 y^2
 - 12636 z^6 y
 + 1119744 z^6)  x^2\\ 
 + (1296 z^5 y^3
 - 101088 z^6 y
 + 17496 z^7)  x\\ 
 + (-648 z^6 y^2
 + 62208 z^7) )\ X^3$ 

\ 

$+(1476 x^{19}
 + (891 y^2
 - 756 z)  x^{16}
 - 16416 z y x^{15}
 + 66078 z^2 x^{14}\\ 
 + (1296 z y^2
 - 10044 z^2)  x^{13}
 - 6804 z^2 y x^{12}\\ 
 + (-2484 z^2 y^2
 - 121932 z^3)  x^{11}
 + (1458 z^2 y^2
 + 32514 z^3 y
 - 109836 z^3)  x^{10}\\ 
 + (142884 z^3 y
 - 60864 z^4)  x^9
 + (-40824 z^3 y^2
 - 86994 z^4)  x^8\\ 
 + (2484 z^3 y^3
 + 17496 z^3 y^2
 + 162432 z^4 y
 + 414072 z^4)  x^7\\ 
 + (-13572 z^4 y^2
 - 396576 z^4 y
 - 181872 z^5)  x^6\\ 
 + (40824 z^4 y^2
 + 20142 z^5 y
 + 456840 z^5)  x^5\\ 
 + (-2160 z^4 y^3
 + 4374 z^4 y^2
 - 50382 z^5 y
 + (-2466 z^6
 - 279936 z^5) )  x^4\\ 
 + (7236 z^5 y^2
 - 82296 z^6)  x^3
 + (36 z^5 y^3
 - 2916 z^5 y^2
 - 9396 z^6 y
 + 186624 z^6)  x^2\\ 
 + (162 z^5 y^3
 - 126 z^6 y^2
 - 10368 z^6 y
 + 16632 z^7)  x\\ 
 + (-162 z^6 y^2
 - 126 z^7 y
 + 10368 z^7) )\ X^2$ 

\ 

$+(-234 y x^{18}
 + 2772 z x^{17}
 + 864 z y x^{15}
 - 22572 z^2 x^{14}
 + 1332 z^2 y x^{13}\\ 
 + (486 z^2 y
 - 2052 z^3)  x^{12}
 + 70956 z^3 x^{11}
 - 11988 z^3 y x^{10}\\ 
 + (654 z^3 y^2
 - 972 z^3 y
 + 23976 z^4)  x^9
 + (-3414 z^4 y
 - 115668 z^4)  x^8\\ 
 + (27216 z^4 y
 + 4812 z^5)  x^7
 + (-1656 z^4 y^2
 - 5832 z^4 y
 - 53460 z^5)  x^6\\ 
 + (5256 z^5 y
 + 93312 z^5)  x^5
 + (120 z^5 y^2
 - 9720 z^5 y
 - 1980 z^6)  x^4\\ 
 + (540 z^5 y^2
 - 588 z^6 y)  x^3
 + (-1080 z^6 y
 + 612 z^7)  x^2\\ 
 + (-36 z^6 y^2
 + 3456 z^7)  x
 + (z^6 y^3
 - 72 z^7 y) )\ X$ 

\ 

$+27 x^{20}
 - 324 z x^{17}
 + 108 z^2 x^{15}
 + 1458 z^2 x^{14}
 - 972 z^3 x^{12}\\ 
 + (18 z^3 y
 - 2916 z^3)  x^{11}
 + 90 z^4 x^{10}
 + 2916 z^4 x^9\\ 
 + (-108 z^4 y
 + 2187 z^4)  x^8
 - 540 z^5 x^7
 + (36 z^5 y
 - 2916 z^5)  x^6\\ 
 + (162 z^5 y
 - 36 z^6)  x^5
 + 810 z^6 x^4
 - 108 z^6 y x^3\\ 
 + (3 z^6 y^2
 + 108 z^7)  x^2
 - 6 z^7 y x
 + 3 z^8$

\end{document}